\documentclass[12pt]{amsart}
\usepackage{amsmath, amsfonts, amssymb, euscript, graphicx, verbatim, amsthm,xy}                       

\usepackage[english]{babel}

\DeclareMathOperator{\pa}{PAut}

\DeclareMathOperator{\dom}{dom}
\DeclareMathOperator{\ran}{ran}
\DeclareMathOperator{\rank}{rank}
\DeclareMathOperator{\rk}{rk}

\newcommand{\pwr}{\,{\wr_p}\,}
\newcommand{\set}[1]{\left\{#1\right\}}
\newcommand{\ra}{\rightarrow}
\newcommand{\isd}{\mathcal{I}_d}
\newcommand{\is}[1]{\mathcal{I}_{#1}}
\newcommand{\abs}[1]{\left\vert#1\right\vert}
\newcommand{\pn}[1]{\mathcal{P}_{#1}}

\theoremstyle{plain}
\newtheorem{prop}{Proposition}[section]

\newtheorem{lem}{Lemma}[section]
\newtheorem{thm}{Theorem}[section]
\theoremstyle{definition}
\newtheorem{defn}{Definition}[section]
\newtheorem{example}{Example}[section]
\newtheorem{rem}{Remark}[section]

\begin{document}
\title{Spectrum of partial automorphisms of regular rooted tree}
\author{Eugenia Kochubinska}

\begin{abstract}
We study properties of eigenvalues of a matrix associated with a randomly chosen partial automorphism of a regular rooted tree. 
We show that asymptotically, as the  numbers of levels goes to infinity, the fraction of non-zero eigenvalues converges to zero in probability.
\end{abstract}
\keywords{partial automorphism, inverse semigroup, eigenvalues, delta measure, random element, uniform distribution}
\subjclass[2010]{20M18, 20M20, 05C05}
\maketitle
\section{Introduction}

We study spectral properties of semigroup of partial automorphisms of a regular $n$-level rooted tree. Here  by a partial automorphism we mean root-preserving  injective tree homomorphism defined on a connected subtree. This semigroup was studied, in particular, in \cite{comb, edm}.

 For the group of partial automorphisms of a regular rooted tree, a similar question was studied in \cite{Evans}. Namely, it was shown that the spectral measure $\Theta$ of a randomly chosen element $\sigma$ of $n$-fold wreath product of symmetric group $\mathcal{S}_d$ converges weakly in probability to the normal Lebesgue measure on the unit circle, i.e, for any trigonometric polynomial $f$,
  $$\lim_{n\to \infty}\mathbb{P}\left\{\int_C f(x)\, \Theta_n(dx)\neq \int f(x)\, \lambda(dx)\right\}=0.$$

 In contrast, for partial automorphisms of a binary rooted tree, in  \cite{eigen-bin} it was shown that the uniform distribution  $\Xi_n$ on eigenvalues of the action matrix converges weakly in probability to $\delta_0$ as $n \to \infty$, where $\delta_0$ is the delta-measure concentrated at $0$. 
 In this article we generalize this result for a regular rooted tree of any degree. Specifically, denote by $B_n=\set{v_i^n\mid i=1, \ldots, d^n}$  the set of vertices of the $n$th level of the $n$-level $d$-regular rooted tree.  To a randomly chosen partial automorphism $y$,  assign the action matrix   $A_y = \left(\mathbf{1}_{\set{x (v_i^n) = v_j^n }}\right)_{i,j=1}^{d^n}.$ 
 Let $$\Xi_n = \frac{1}{d^n} \sum\limits_{k=1}^{d^n} \delta_{\lambda_k}$$ be the uniform distribution on eigenvalues of $A_y$. We show that $\Xi_n$ converges weakly in probability to $\delta_0$ as $n \to \infty$, where $\delta_0$ is the delta-measure concentrated at $0$.
 
 The remaining of the paper is organized as follows. Section 2 contains basic facts on partial wreath product of semigroup and its connection with a semigroup of partial automorphisms of a regular rooted tree. The main result is stated and proved in Section~3.

\section{Preliminaries}\label{sec:basic}
For a set $X=\set{1,2, \ldots, d}$ consider the set $\is{d}$ of all partial bijections.  This set forms an inverse semigroup under natural composition law, namely, $f \circ g:
\dom(f) \cap f^{-1} \dom(g) \ni x \mapsto g(f(x))$ for $f, \; g \in \is{d}$.  Detailed description of this semigroup one can  find in \cite[Chapter 2]{GM}.

Recall the definition of a partial wreath product of semigroups.  Let $S$ be an arbitrary semigroup. For functions $f\colon \dom(f) \ra S$ and $g\colon \dom(g)\ra S$ define 
 the product $fg$  as:
\begin{equation*}
\dom(fg)=\dom(f)\cap\dom(g),\  (fg)(x)=f(x)g(x) \text{\ for all } x\in
\dom(fg).
\end{equation*}
For $a\in \isd, f\colon \dom(f)\ra S$, define $f^a$ as:
\begin{equation*}
\begin{gathered}
(f^a)(x)=f(x^a),\ \dom(f^a)=\{x \in \dom(a); x^a\in \dom(f)\}.
\end{gathered}
\end{equation*}

\begin{defn}
The partial wreath square of semigroup $\isd$ is  the
set $$\set{(f,a)\mid a \in \isd, f\colon  \dom(a)\ra \isd }$$ with
composition defined by
$$(f,a)\cdot (g,b)=(fg^a,ab)$$
Denote it by $\isd \pwr \isd$.
\end{defn}

It is well-known that the partial wreath square of inverse semigroup is an inverse semigroup \cite[Lemmas 2.22 and 4.6]{Meldrum}. We
may recursively define any  partial wreath power of the finite
inverse symmetric semigroup.  Denote by $\pn{n}$ the $n$th partial wreath power of $\is{d}$.

\begin{defn}
The partial wreath $n$-th power of semigroup $\isd$ is defined as
a semigroup $$\pn{n}=\big(\pn{n-1} \big) \pwr \isd =\set{(f,a)\mid a \in \isd, \ f\colon \dom(a)\ra \pn{n-1}}$$ with composition
defined by
$$(f,a) \cdot (g,b)=(fg^a,ab),$$ where  $\pn{n-1}$ is the
partial wreath $(n-1)$-th power of semigroup $\isd$
\end{defn}


\begin{prop}\label{prop:card}\rm{\cite{comb}}
The number of elemets in the semigroup $\pn{n}$ is equal to
\begin{equation}\label{eq:number}
N_n=\sum_{a\in \isd}(N_{n-1})^{\rank (a)}=S(N_{n-1})=\underset{n}{\underbrace{S(S\ldots (S}}(1))\ldots),
\end{equation}
where $S(x)=\sum_{i=1}^d \binom{d}{i}^2
i!x^i$.
\end{prop}

\begin{rem}
Let $T$ be an $n$-level $d$-regular rooted tree. We define a partial automorphism of a tree $T$ as an isomorphism  $y: T_1\to T_2$ between its subtrees $T_1$ and $T_2$ containing root. Denote $\dom(y):=T_1$, $\ran (y):= T_2$ domain and image of $y$ respectively. Let $\pa T$ be the set of all partial automorphisms of $T$. Obviously, $\pa T$ forms a semigroup under natural composition law.  It was proved in \cite[Theorem 1]{comb} that the  partial wreath power $\pn{n}$ is isomorphic to $\pa T$. 
\end{rem}

\section{Asymptotic behavior of a spectral measure of a regular rooted tree}

Let $T$ be an $n$-level $d$-regular rooted tree and $\pa T$ be its semigroup of partial automorphisms. We identify $y\in \pn{n}$ with a partial automorphism from $\pa T$.
  Let $B_n$ be the set of vertices of the $n$th level of $T$. It is clear that $\abs{B_n} =d^n$. 

Let us enumerate the vertices of $B_n$ by positive integers from  1 to $d^n$: $$B_n=\set{v_i^n\mid i=1, \ldots, d^n}.$$ To a randomly chosen partial automorphism $y \in \mathcal{P}_n$, we assign the matrix   $$A_y = \left(\mathbf{1}_{\set{y (v_i^n) = v_j^n }}\right)_{i,j=1}^{d^n}.$$ In other words, $(i, j)$th entry of $A_y$ is equal to~1, if a transformation $y$\ maps  $i$ to $j$, and~0, otherwise.
\begin{rem}
In an automorphism group of a tree such a matrix describes completely the action of an automorphism. Unfortunately, for a semigroup this is not the case, see the example below. That is why we cannot use technique developed by Evans in \cite{Evans}. Also, the generalization of result from \cite{eigen-bin} is not straightforward, despite the result is similar.
\end{rem} 
\begin{example}\label{example:matrix} Consider a  binary tree and the partial automorphism  $y \in \mathcal{P}_2$, which  acts in the following way
	$$
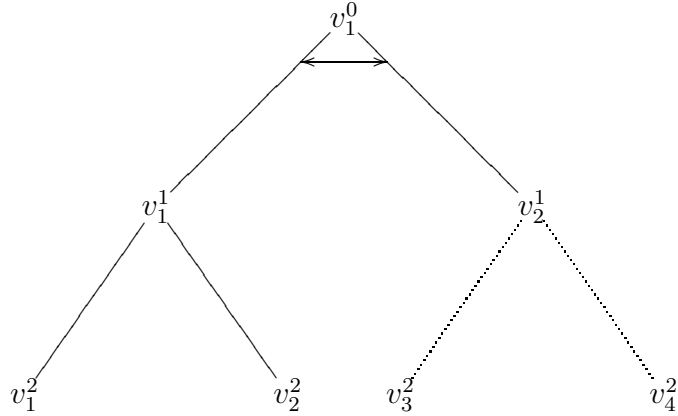
\begin{xy}
	0;<2.5cm,0cm>:<0cm,2.5cm>::
	(2.2,2)*{v_1^0}="v0";
	(1.2,1)*{v_1^1}="v11" **@{-}; ?(.23)*{}="l";
	"v0"; (3.2,1)*{v_2^1}="v12" **@{-}; ?(.23)*{}="v";
	"v11"; (0.5,0)*{v_1^2} **@{-}; 
	"v11"; (1.9,0)*{v_2^2} **@{-}; 
	"v12"; (2.5,0)*{v_3^2} **@{.}; 
	"v12"; (3.9,0)*{v_4^2} **@{.}; 
	"l"; "v" **@{-} ?>* @{>};  ?<* @{<}
	\end{xy}$$
(dotted lines mean that these edges are not in domain of $y$).	
Then the corresponding matrix for $x$ is
	$$A_{y}=\begin{pmatrix}
	0&0&1&0\\
	0&0&0&1\\
    0&0&0&0\\
    0&0&0&0
	\end{pmatrix}.$$
	
We can see from $A_y$ that $y$  maps $v_1^1$ to $v_2^1$. However, from $A_y$ we cannot infer the action of $y$ on $v_2^1$, in particular, whether $v_2^1$ belongs to the domain of $y$. 
\end{example}

Let $\chi_y(\lambda)$ be the characteristic polynomial of $A_y$ and $\lambda_1, \ldots, \lambda_{d^n}$ be its roots, respecting multiplicity. Denote $$\Xi_n = \frac{1}{d^n} \sum\limits_{k=1}^{d^n} \delta_{\lambda_k}$$ the uniform distribution on eigenvalues of $A_y$.

Our main result is the following theorem.

\begin{thm}\label{thm:main}
For any function $f \in C(D)$, where $D= \{z \in \mathbb{C}\mid |z|\leq 1\}$ is a unit disc, \begin{equation}\label{eq:main}
\int_D f(x)\,\Xi_n (dx) \overset{\mathbb P}{\longrightarrow} f(0),\ \ \ n\rightarrow \infty.\end{equation}
In other words, $\Xi_n$ converges weakly in probability to $\delta_0$ as $n \to \infty$, where $\delta_0$ is the delta-measure concentrated at $0$.\end{thm}

\begin{proof} For a partial automorphism $y\in \pn{n}$, let $\eta_n(y)=\Xi_n(0)$ be the fraction of zero eigenvalues of $A_y$, and let $\xi_n(y) = 1-\eta_n(y)$ denote a fraction of non-zero eigenvalues. We have to prove that for a randomly chosen partial automorphism $y\in \mathcal{P}_n$ $$\eta_n(y) \overset{\mathbb P}{\longrightarrow} 1,\quad n\to\infty,
$$ or, equivalently,
$$\xi_n(y) \overset{\mathbb P}{\longrightarrow} 0,\quad n\to\infty.$$

Thanks to the Markov inequality, it is enough to show that $$\mathbb{E}\xi_n(y) \to 0, \quad n\to \infty.$$

For $y\in \mathcal{P}_n$, denote
$$
S(y) = \set{j: v^n_j\in \dom y^m\text{ for all }m\ge 1}
$$
the indices of vertices of the bottom level, which ``survive'' under the action of $y$, and define the \textit{ultimate rank} of $y$ by $\rk(y) = \# S(y)$. Then 
$$
\xi_n(y) = \frac{\rk(y)}{d^n},
$$
whence
$$
\mathbb{E}\xi_n(y)=\dfrac{R_n}{d^n N_n},
$$
where $R_n=\sum_{y\in \mathcal{I}_n}\rk(y)$.

Also define $\rank y= \# \left(\dom y\cap V^n\right)$. 

Recalling that 
$$
\mathcal P_n = \mathcal P_{n-1}\pwr \isd,
$$
we can identify $y\in \mathcal P_n$ with an element $a_y\in\isd$ and a collection $\{y_x\in \mathcal P_{n-1},x\in \dom a\}$.

We can write 
$$
R_n = \sum_{a\in\isd} R_n(a),\text{ where }R_n(a) =\sum_{y\in \mathcal P_n: a_y = a} \rk y.
$$


A partial transformation $a\in \isd$ is a product of disjoint cycles $(x_1\ldots x_k)$ and chains $[x_1\ldots x_k]$, that is $a(x_i) =x_{i+1}$, $1\le i\le k-1$ and $x_k\notin \dom a$. 

If $x$ belongs to a chain, then no elements survive under $x$.

If $x=x_1$ belongs to a cycle $(x_1,\dots,x_k)$, then the number of elements surviving under $x$ is 
$$
\rk (y_{x_1} \cdots y_{x_k}).
$$
As a result, if $a$ contains cycles $(x_{i1},\dots,x_{ic_i})$, $i=1,\dots r$, then
$$
\rk y = \sum_{i=1}^r c_i \rk\left( y_{x_{i1}}\cdots y_{x_{ic_i}}\right).
$$

Therefore, 
\begin{equation}\label{eq:Rn(a)}
R_n(a) = \sum_{i=1}^r c_i \sum_{y_{1},\dots,y_{c_i}\in \mathcal P_{n-1}} \rk\left(y_1\cdots y_{c_i}\right).
\end{equation}

For a convenience the rest of the proof is split into lemmas.

\begin{lem}\label{lem:lem1}
	For $n\geq 2$ $$R_n(a) \leq R_{n-1}\left(\frac{N_{n-1}}{d}\right)^{\rank (a)-1}\rank(a).$$
\end{lem}
\begin{proof}
The element $y_1$ can be decomposed into a product of idempotent $e_{y_1}$ on the domain of $y_1$ and an automorphism $\sigma_{y_1}$. Then
\begin{gather*}
\sum_{y_{1},\dots,y_{c_i}\in \mathcal P_{n-1}} \rk\left(y_1\cdots y_{c_i}\right)  = \sum_{y_{1},\dots,y_{c_i}\in \mathcal P_{n-1}} \rk\left(e_{y_1} \sigma_{y_1}y_2\cdots y_{c_i}\right).
\end{gather*}

Since $\sigma_{y_1}$ is an automorphism, then $\sigma_{y_1}y_2$ is a bijection on $\mathcal P_{n-1}$, so
\begin{gather*}
\sum_{y_{1},\dots,y_{c_i}\in \mathcal P_{n-1}} \rk\left(e_{y_1} \sigma_{y_1}y_2\cdots y_{c_i}\right) = \sum_{y_{1},\dots,y_{c_i}\in \mathcal P_{n-1}} \rk\left(e_{y_1}y_2\cdots y_{c_i}\right). 
\end{gather*}
Further $S(e_{y_1}y_2\cdots y_{c_i})\subset\dom y_1\cap S(y_2\cdots y_{c_i})$, therefore,
\begin{gather*}
\sum_{y_{1},\dots,y_{c_i}\in \mathcal P_{n-1}} \rk\left(e_{y_1}y_2\cdots y_{c_i}\right) \le \sum_{y_{1},\dots,y_{c_i}\in \mathcal P_{n-1}} \sum_{k=1}^{d^n}\mathbf{1}_{v^n_k\in \dom y_1} \mathbf{1}_{v_k^n \in S(y_2\cdots y_{c_i})}.
\end{gather*}

By symmetry, the sum $\sum_{y_1\in \mathcal P_{n-1}} \mathbf{1}_{v^n_k\in \dom y_1}$ does not depend on $k$. Hence \begin{gather*}\sum_{y_1\in \mathcal P_{n-1}} \mathbf{1}_{v^n_k\in \dom y_1}=\frac{1}{d^n}\sum_{j=1}^{d^n}\sum_{y_1\in \mathcal P_{n-1}}\mathbf{1}_{v^n_j\in \dom y_1}\\
=\frac{1}{d^n}\sum_{y_1\in \mathcal P_{n-1}}\sum_{j=1}^{d^n}\mathbf{1}_{v^n_j\in \dom y_1}= \frac{1}{d^n} \sum_{y\in \mathcal P_{n-1}} \rank(y) =: \frac{1}{d^n}R'_{n-1}.
\end{gather*}

Consequently,
\begin{gather*}
\sum_{y_{1},\dots,y_{c_i}\in \mathcal P_{n-1}} \rk\left(y_1\cdots y_{c_i}\right) \le \frac{R'_{n-1}}{d^n}  \sum_{y_{2},\dots,y_{c_i}\in \mathcal P_{n-1}} \rk\left(y_2\cdots y_{c_i}\right)\\
\le\dots \leq \left(\frac{R'_{n-1}}{d^n}\right)^{c_{i}-1} \sum_{y_{c_i}\in \mathcal P_{n-1}} \rk\left(y_{c_i}\right) = \left(\frac{R'_{n-1}}{d^n}\right)^{c_{i}-1} R_{n-1}.
\end{gather*}

Note that $R'_{n-1}\leq d N_{n-1}$, since rank of every element from $\mathcal{P}_{n-1}$ is not greater than $d^{n-1}$. 
Combining this with the above inequality, we get for $n\geq 2$
\begin{align*}
R_n(a)&\leq \sum_{i=1}^r c_i\left(\frac{R'_{n-1}}{d^n}\right)^{c_i-1}R_{n-1}\leq R_{n-1}\sum_{i=1}^r c_i\left(\frac{d^{n-1}N_{n-1}}{d^n}\right)^{c_i-1}\\
&=R_{n-1}\sum_{i=1}^r c_i \left(\frac{N_{n-1}}{d}\right)^{c_i-1}\leq R_{n-1}\sum_{i=1}^r c_i \left(\frac{N_{n-1}}{d}\right)^{\rank (a)-1}\\&
=R_{n-1}\left(\frac{N_{n-1}}{d}\right)^{\rank (a)-1}\sum_{i=1}^rc_i=R_{n-1}\left(\frac{N_{n-1}}{d}\right)^{\rank (a)-1}\rank(a).\qedhere
\end{align*}
\end{proof}

\begin{lem}\label{lem:lem2} For $n \to \infty$
	$$
	\frac{R_n}{d^n N_n}\leq r_n \frac{R_{n-1}}{d^{n-1} N_{n-1}}
	$$
	with $\limsup_{n\to\infty} r_n\leq 1/d$.
\end{lem}	

\begin{proof} Using Lemma~\ref{lem:lem1}, we get
	\begin{align*}
	\frac{R_n}{d^n N_n}&=\frac{1}{d^n N_n}\sum_{a\in \isd} R_n(a)\leq \frac{1}{d^n N_n}\sum_{a\in \isd} R_{n-1}\left(\frac{N_{n-1}}{d}\right)^{\rank (a)-1}\rank(a)\\&=\frac{R_{n-1}}{d^n N_n}\sum_{a\in \isd} \left(\frac{N_{n-1}}{d}\right)^{\rank (a)-1}\rank(a)=r_n\frac{R_{n-1}}{d^{n-1} N_{n-1}},
	\end{align*}
	where $$r_n=\frac{1}{dN_n}\sum_{a\in \isd} \frac{N_{n-1}^{\rank (a)}}{d^{\rank (a)-1}}\rank(a)\leq \frac{1}{dN_n}\sum_{a\in \isd} (N_{n-1})^{\rank (a)}=\frac{1}{d}.$$
	Here we have  used \eqref{eq:number} and the fact that for any $a\in \isd$, $\rank(a)\leq d^{\rank(a)-1}$.
\end{proof}

 As a result, 
$$
\frac{R_n}{d^n N_n}\to 0,\quad n\to\infty,
$$
exponentially fast. The proof of Theorem~\ref{thm:main} is now immediate.
\end{proof}


\begin{thebibliography}{70}
	\bibitem{Evans}  S.~N. Evans   \emph{Eigenvalues of random wreath products} Electron. J. Probability. \textbf{7} (9) (2002),  1--15.
	
	\bibitem{GM} O. Ganyushkin, V. Mazorchuk. Classical Finite Transformation Semigroups: An Introduction. Springer,  2008. 
	
	\bibitem{comb} E. Kochubinska. \emph{Combinatorics of partial wreath power of finite inverse
	symmetric semigroup $\mathcal{IS}_d$}, Algebra and Discrete
	Mathematics 1(2007), 49--61.
	
	\bibitem{eigen-bin} E. Kochubinska. \emph{Spectral properties of partial automorphisms of  binary rooted tree}, \emph{Algebra and Discrete
		Mathematics} \textbf{26}(2) (2018), 280--289.
	
	\bibitem{edm} E. Kochubinska. \emph{On cross-sections of partial wreath product of inverse semigroups}, Electron. Notes Discrete Math. \textbf{28}, (2007), 379--386.
	
	\bibitem{Meldrum} J. D. P. Meldrum. Wreath product of groups and semigroups. Pitman Monographs and Surveys in Pure and Applied Mathematics, 74. Harlow: Longman Group Ltd,  1995.
	


	

	

\end{thebibliography}
\end{document}